\documentclass[12pt]{amsart}
\usepackage[utf8]{inputenc}
\usepackage[T1]{fontenc}
\usepackage{comment}
\usepackage{a4wide}
\usepackage{scrextend}

\usepackage{amssymb,amsthm,latexsym}
\usepackage{esint}
\usepackage{newtxtext,newtxmath} 
\usepackage{mathrsfs}
\usepackage{dsfont}
\usepackage[mathscr]{euscript}
\usepackage{esint}
\usepackage{newtxtext,newtxmath} 
\usepackage[usenames,dvipsnames,x11names]{xcolor}

\usepackage{url}

\newtheorem{theorem}{Theorem}[section]
\newtheorem{lemma}[theorem]{Lemma}

\newtheorem{proposition}[theorem]{Proposition}

\newcounter{maintheorem}

\newtheorem{mainth}[maintheorem]{Theorem}
\theoremstyle{definition}
\newtheorem*{remark*}{Remark}

\numberwithin{equation}{section}

\newcommand{\vertiii}[1]{{\left\vert\kern-0.25ex\left\vert\kern-0.25ex\left\vert #1 
    \right\vert\kern-0.25ex\right\vert\kern-0.25ex\right\vert}}

\newcounter{smallromans}

{\end{list}}

\newcounter{smallromansdash}

{\end{list}}

\newcounter{bigromans}

{\end{list}}

\makeatother

\newcommand{\Htheta}{\operatorname{H}(\theta)}
\newcommand{\s}{\subseteq}

\title{Discrete subgroups of normed spaces are free}

\author[T.~Kania]{Tomasz Kania}
\address{Institute of Mathematics, Czech Academy of Sciences, \v{Z}itn\'{a} 25, 115~67 Prague 1, Czech Republic}
\email{tomasz.marcin.kania@gmail.com}
\author[Z. Kostana]{Ziemowit Kostana}
\address{Institute of Mathematics, Czech Academy of Sciences, \v{Z}itn\'{a} 25, 115~67 Prague 1, Czech Republic}
\email{kostana@math.cas.cz}

\thanks{Institute of Mathematics, Czech Academy of Sciences; RVO: 67985840. Research of Z. Kostana was supported by the GA\v{C}R project EXPRO 20-31529X}
\date{\today}

\begin{document}
\begin{abstract}
    Ancel, Dobrowolski, and Grabowski (\emph{Studia Math., 1994)} proved that every countable discrete subgroup of the additive group of a normed space is free Abelian, hence isomorphic to the direct sum of a certain number of copies of the additive group of the integers. In the present paper, we take a set-theoretic approach based on the theory of elementary submodels and the Singular Compactness Theorem to remove the cardinality constraint from their result and prove that indeed every discrete subgroup of the additive group of a normed space is free Abelian. 
\end{abstract}
\maketitle

\section{Introduction} 
It is a familiar fact from Euclidean geometry that every discrete subgroup of $\mathbb R^n$ is free Abelian of rank at most $n$ with a freely generating set of linearly independent vectors in $\mathbb R^n$. (An Abelian group is \emph{free} whenever it is a free object in the category of Abelian groups (\emph{i.e.}, $\mathbb Z$-modules); being free $\mathbb Z$-modules, free Abelian groups have \emph{bases}, that is freely generating sets, and as such they are necessarily isomorphic to direct sums $\mathbb Z^{(\Gamma)}$ of the group of integers for some index set $\Gamma$; the cardinality of any basis of a free Abelian group is called the \emph{rank}.) Just as vector spaces, sometimes free Abelian groups do come equipped with canonical or otherwise evident bases. Nonetheless, as for vector spaces, sometimes this is not the case, just to mention the groups $B(\Gamma, \mathbb Z)$ comprising bounded integer-valued functions on a set $\Gamma$ that are free Abelian (this is the Specker--N\"obeling theorem; \cite{nobeling1968verallgemeinerung, specker1950}). The reason behind this obstacle can be formally delineated from the point of view of (generalised) computability theory (\cite{Greenberg2018finding}).\smallskip

It is claimed without proof on p.~58 in \cite{banaszczyk2006additive} that the discrete subgroup of the Banach space $L_\infty[0,1]$ comprising integer-valued functions is not free. Unfortunately, this claim is \emph{not} correct in the light of Bergman's result (\cite{bergman1972boolean}; see also \cite{hill1973additive} for an alternative proof) asserting that the additive group of a commutative ring generated by idempotents is a direct sum of cyclic groups, so free Abelian if torsion-free.\smallskip

Since being free is intimately related to the notion of a basis, one can observe that free Abelian groups are only weakly second-order characterisable, hence not directly amenable to absoluteness arguments (see \cite[Example 7.3.4]{vaananen2023}). (A model of a finite vocabulary is weakly second-order characterisable if it satisfies a second-order sentence that has at most one model in each cardinality, up to isomorphism.) Indeed, Shelah \cite{shelah1974free} famously proved that the problem of whether all Whitehead groups are free is independent of ZFC (we refer to the said paper for further details). \smallskip

The Axiom of Choice implies that subgroups of free Abelian groups are free \cite[Appendix 2 \S2, p.~880]{lang2022algebra}, however, there exist non-free Abelian groups whose all countable groups are free; the Baer--Specker group, that is the product of countably many copies of the group of integers, is a~notable example. This gives rise to the notion of a $\lambda$-\emph{free Abelian group}, that is a group whose all subgroups of cardinality smaller than $\lambda$ are free Abelian. Eklof \cite{eklof1975onthe} proved that for every regular cardinal $\lambda$ there exists a $\lambda$-free Abelian group that is not $\lambda^+$-free. This result is sharp as we have (the special case of) Shelah's Singular Compactness Theorem (\cite[Ch.~IV Thm.~3.5]{eklof2002}).
\begin{theorem}\label{singularcomp}
    Let $\kappa$ be a singular cardinal. If $A$ is a $\kappa$-free Abelian group of size $\kappa$, then $A$ is free Abelian.
\end{theorem}

Discrete subgroups of normed spaces are closed. Countable discrete subgroups of (the additive groups of) normed spaces are free Abelian (\cite[Theorem 1.1]{ancel1994closed}); the proof presented by Acel, Dobrowolski, and Grabowski is inductive and would break down at the $\omega$\textsuperscript{th} step should one wish to continue the process transfinitely. Indeed, it is a natural question of whether discrete subgroups of normed spaces are free Abelian, hence a potential proof based on transfinite induction seems appealing.  The first-named author is indebted to Piotr Niemiec for making us aware of this problem. Quite recently, the same question has been asked by van Gent (\cite[Question 3.14]{van2021indecomposable}) in a more restricted setting of Hilbert spaces. The main result is the affirmative answer (in ZFC) to the said question.

\begin{mainth}\label{mainthm}
    Discrete subgroups of normed spaces are free Abelian.
\end{mainth}

\section{Preliminaries}

\subsection*{Elementary submodels \& clubs} 
Let $\kappa$ be a regular cardinal. A set $D \s \kappa$ is \emph{closed unbounded} (or a~\emph{club}), whenever it is unbounded and closed in the order topology. We write $\operatorname{acc} D$ for the set of \emph{accumulation points} of $D$\footnote{This definition is, of course, meaningful for any subset of $\kappa$ not only a club.}, \emph{i.e.},
\[
    \operatorname{acc} D:= \{\beta<\kappa \mid \forall \, \alpha<\beta \quad D \cap (\beta \setminus \alpha) \neq \emptyset\}.
\]
If $D$ is a club, then $\operatorname{acc} D \s D$.

Using chains of elementary submodels is by now a well-established tool in many parts of mathematical logic. We recall here basic definitions, and refer the reader to \cite{dow1988introduction} for more elaborated introduction, together with some applications.

Let $\theta$ be a cardinal number. The structure $\Htheta$ is the set of all \emph{hereditarily $<\kappa$-sets}, that is, sets whose {transitive closure} (the smallest, with respect to inclusion, transitive set that contains a given set) has cardinality less than $\theta$. Whenever $\theta$ is uncountable and regular, $\Htheta$ is a model of all the axioms of $\operatorname{ZFC}$, possibly except the Axiom of Power Set~\cite[Ch.~IV, Thm. 6.5]{kunen}. We write $M \prec \Htheta$ when $M = (M, \in)$ is an elementary substructure of $(\Htheta, \in)$. It is often useful to work with specific sequences of such elementary submodels. Therefore, for a cardinal $\kappa$, we say that a~family $\{ M_\alpha \mid \alpha < \kappa \}$ is a \emph{continuous chain of elementary submodels} of $\Htheta$, whenever
\begin{enumerate}
    \item $\forall \alpha < \kappa \quad M_\alpha \prec \Htheta$,
    \item $\forall \alpha < \kappa \quad M_\alpha \in M_{\alpha+1}$,
    \item $\forall \beta \in \operatorname{acc} \kappa \quad M_\beta = \bigcup_{\alpha < \beta} M_\alpha$.
\end{enumerate}

The next proposition is essentially an instance of a well-known fact, that we can always produce continuous chains of elementary submodels by repeatedly taking closures under Skolem functions.

\begin{proposition}\label{contchains}
    Assume $\lambda<\theta$ are uncountable, regular cardinals, and $Z \in \Htheta$ is an arbitrary set. Then there exists a continuous chain of elementary submodels of $\Htheta$, $\{M_\alpha \mid \alpha< \lambda\}$, such that:
    \begin{enumerate}
        \item $Z \in M_0$,
        \item $\forall \alpha<\lambda \quad \alpha \s M_\alpha$,
        \item $\forall \alpha<\lambda \quad |M_\alpha|=|\alpha|+\aleph_0$.
    \end{enumerate}
    Moreover, the set $\{M_\alpha \cap \lambda \mid \alpha<\lambda\}$ is a club subset of $\lambda$.
\end{proposition}
\begin{proof}
    We define the chain $\{M_\alpha \mid \alpha < \lambda\}$ recursively. 
    \begin{itemize}
        \item We declare $M_0$ to be the closure of $\{Z\}$ under Skolem functions of $\Htheta$.
        \item If $M_\alpha$ is defined, we declare $M_{\alpha+1}$ to be closure of $M_\alpha \cup \{M_\alpha\}$ under Skolem functions of $\Htheta$. 
        \item For a limit ordinal $\beta<\lambda$, we declare $M_\beta:= \bigcup_{\alpha<\beta}M_\alpha$. 
    \end{itemize} It is standard to check that $\{M_\alpha \mid \alpha < \lambda\}$ is as required.
\end{proof}

\subsection*{Abelian groups} All groups under consideration are assumed to be Abelian. Terminology related to groups is standard; we refer to \cite{fuchs1970} for all unexplained terms. The following folklore fact is well-known; we provide proof for the sake of completeness.

\begin{proposition}\label{fact}
    Suppose that $A$ is a (free) subgroup of a free Abelian group $B$. Then $B \slash A$ is free precisely when every basis of $A$ may be extended to a basis of $B$.
\end{proposition}
\begin{proof}
    If $B \slash A$ is free, then by \cite[Thm. 14.4]{fuchs1970}, $A$ is a direct summand of $B$, and the result follows. 

    In the other direction, suppose $E$ is a basis of $A$, and let $F \supseteq E$ be its extension to a basis of $B$. A moment's thought shows that the set $\{f + A \mid f \in F\setminus E\}$ constitutes a basis of $B\slash A$.
\end{proof}

The following lemma is also standard (see \cite{eklof2002}). For the reader's convenience, we include the proof too.
\begin{lemma}\label{clublemma}
    Suppose that $\kappa$ is an uncountable, regular cardinal and that $A$ is a $\kappa$-free Abelian group of rank $\kappa$. Then the following conditions are equivalent:
    \begin{enumerate}
        \item $A$ is free,
        \item if there is a representation $A={\bigcup_{\alpha<\kappa}A_\alpha}$ as a continuous chain of subgroups, each having size less than $\kappa$, the set
       $
            \{\alpha<\kappa \mid A \slash A_\alpha \text{ is $\kappa$-free}\}
       $
        contains a club subset of $\kappa$.
    \end{enumerate}
\end{lemma}
\begin{proof}
    Suppose that $A$ is free and let $\{b_\alpha \mid \alpha < \kappa\}$ be a basis of $A$. For $\alpha < \kappa$ set $A_\alpha = \langle \{b_\beta \mid \beta < \alpha\}\rangle$ so each $A_\alpha$ is free and so is $A/A_\alpha$ (hence $\kappa$-free) since $A_\alpha$ is a direct summand in $A$. We may thus take $D = \kappa$, which is certainly a club. 

    For the other direction, fix a representation $A={\bigcup_{\alpha<\kappa}A_\alpha}$, and assume that there exists a club $D \subseteq \{\alpha<\kappa \mid A \slash A_\alpha \text{ is $\kappa$-free}\}$. We define recursively a trasfinite sequence $\{E_\alpha \mid \alpha \in D\}$, satisfying:
    \begin{itemize}
        \item $\forall \, \alpha \in D \quad E_\alpha \text{ is a basis of $A_\alpha$}$,
        \item $\forall\, \{\alpha,\beta\} \in [D]^2 \quad \alpha<\beta \implies E_\alpha = E_\beta \cap A_\alpha$,
        \item $\forall\, \beta \in \operatorname{acc} D \quad E_\beta={\bigcup_{\alpha<\beta}E_\alpha}.$
    \end{itemize}

    The limit stage of the construction is straightforward. In the successor step, assume that $E_\alpha$ is known, and $\beta= \min{D\setminus(\alpha+1)}$. By the choice of $D$, the group $A_\beta \slash A_\alpha$ is free, so $E_\alpha$ can be extended to a basis $E_\beta$ of $A_\beta$. Finally, we observe that ${\bigcup_{\alpha \in D} E_\alpha}$ is a basis of $A$.
\end{proof}


\section{Proof of Theorem~\ref{mainthm}}


The next simple lemma is the key ingredient in the proof of Theorem \ref{mainthm}.

\begin{lemma}\label{mainlemma}
Let $\theta \geqslant |X|$ be an uncountable regular cardinal. Suppose that $A,X \in M \prec \Htheta$. Then $A\slash A\cap M$ is a discrete subgroup of $X \slash \overline{X\cap M}$, in the topology generated by the norm of the quotient space, i.e.,
\[
    \begin{array}{lcl}\| y + \overline{X\cap M}\|_{X \slash \overline{X\cap M}} &= & \inf\{\|y - x\|_X \mid x \in \overline{X\cap M}\}\\
    & = & \inf\{\|y - x\|_X \mid x \in X\cap M\}.
    \end{array}
\]

\end{lemma}
\begin{proof}
    It is evident that $A\slash A\cap M$ is a subgroup of $X\slash \overline{X\cap M}$, so it is sufficient to show that it is discrete. Let $\delta$ be a positive, rational number, strictly below $\inf\{\|a\|_X \mid a\in A \setminus \{0\}\}$. The proof will be concluded if we show that for any $a \in A$, the inequality
    \[
        \|a +\overline{X\cap M}\|_{X \slash \overline{X\cap M}}<\delta\slash 3
    \]
    implies that $a \in M$. 

    Indeed, fix a non-zero element $a \in A$, and suppose that $\|a +\overline{X\cap M}\|_{X \slash \overline{X\cap M}}<\delta\slash 3$. Let $x\in X \cap M$ be an element witnessing it. By elementarity, we may find $a' \in A\cap M$, such that $\|x-a'\|_X<\delta \slash 3$.
    It follows that 
    \[
        \|a - a'\|_X \leqslant \| a-x\|_X + \|x-a'\|_X \leqslant \frac{2}{3}\cdot \delta,
    \]
    so $a=a' \in A\cap M$.
\end{proof}
We are now ready to prove Theorem~\ref{mainthm}.
\begin{proof}[Proof of Theorem \ref{mainthm}]
    The proof is by induction on $|A|$.
    \begin{itemize}
        \item $|A| \leqslant \aleph_0$. This case is well-known (\cite[Theorem 1.1]{ancel1994closed}).
        \item $|A|=\kappa^+$. Suppose that the Theorem holds for all groups $A$ of size at most $\kappa$ and infinite $\kappa$. Let $\theta$ be a cardinal from Lemma \ref{mainlemma}. By Proposition \ref{contchains}, we can fix a continuous chain $\{M_\alpha \mid \alpha< \kappa^+\}$ of elementary submodels of $\Htheta$, satisfying for all $\alpha<\kappa^+:$
        \begin{itemize}
            \item $|M_\alpha|\le\kappa$,
            \item $\alpha \subseteq M_\alpha.$
        \end{itemize}

        Given any $\alpha<\kappa^+$, by Lemma \ref{mainlemma} the quotient 
        \[
            A\slash (A\cap M_\alpha) \leqslant X \slash \overline{X\cap M_\alpha}
        \]
        is discrete. By the inductive hypothesis, $A\slash (A\cap M_\alpha)$ is $\kappa^+$-free. As the set $\{M_\alpha \cap \kappa^+ \mid \alpha<\kappa^+\}$ is a club, by Lemma \ref{clublemma} it follows that $A$ is free.
        \item $|A|$ singular. This case is handled at once by Theorem \ref{singularcomp}.
        \item $|A|=\kappa$ is weakly inaccessible, that is, $\lambda < \kappa$ implies $\lambda^+ < \kappa$. The proof is very much like in the successor case. Let $\theta$ be a cardinal from Lemma \ref{mainlemma}. Applying Proposition \ref{contchains}, we fix a continuous chain $\{M_\alpha \mid \alpha< \kappa\}$ of elementary submodels of $\Htheta$, satisfying for all $\alpha<\kappa:$
        \begin{itemize}
            \item $|M_\alpha|=|\alpha|+\aleph_0$,
            \item $\alpha \subseteq M_\alpha.$
        \end{itemize}
        By Lemma \ref{mainlemma}, for any $\alpha<\kappa$ the quotient 
        \[
            A\slash (A\cap M_\alpha) \leqslant X \slash \overline{X\cap M_\alpha}
        \]
        is discrete. Consequently, by the inductive hypothesis, every such quotient is $\kappa$-free. Since the set $\{M_\alpha \cap \kappa \mid \alpha<\kappa\}$ is a club, the group $A$ is free.
    \end{itemize}
\end{proof}
\subsection*{Acknowledgements} We are indebted to Piotr Niemiec (Jagiellonian University) for communicating the question behind the main result of the paper. We are also grateful to James E. Hanson (University of Maryland) for explaining the absoluteness aspects of freeness and Whiteheadness that convinced us to abandon our strategies towards a negative solution to the problem outside ZFC.

\label{sec:annex-1:-list}
  \bibliographystyle{plain}%
  \bibliography{bibliography}

\end{document}